\UseRawInputEncoding

\documentclass[12pt]{amsart}
\usepackage{amsfonts}
\usepackage{amsmath}
\usepackage{amssymb,latexsym}
\usepackage[mathcal]{eucal}
\usepackage[pdftex,bookmarks,colorlinks,breaklinks]{hyperref}

\input xy
\xyoption{all}

\newcommand{\Fcal}{\mathcal{F}}

\newcommand{\Rcal}{\mathcal{R}}

\newcommand{\Wcal}{\mathcal{W}}
\newcommand{\Xcal}{\mathcal{X}}

\newcommand{\bX}{\mathbf{X}}

\newcommand{\Z}{\mathbb{Z}}
\newcommand{\Q}{\mathbb{Q}}

\newcommand{\N}{\mathbb{N}}

\newcommand{\A}{\mathbb{A}}

\newcommand{\del}{\delta}

\newcommand{\sig}{\sigma}

\newcommand{\la}{\lambda}

\newcommand{\om}{\omega}
\newcommand{\Om}{\Omega}

\newcommand{\ol}{\overline}

\newcommand{\br}{\vspace{3 mm}}

\newcommand{\Aut}{{\rm{Aut\,}}}

\newcommand{\id}{{\rm{id}}}

\newcommand{\Homeo}{\rm{Homeo}}

\swapnumbers
\theoremstyle{plain}
\newtheorem{thm}{Theorem}[section]

\newtheorem{prop}[thm]{Proposition}

\theoremstyle{definition}

\newtheorem{rmk}[thm]{Remark}

\newtheorem{question}[thm]{Question}

\begin{document}


\title[Rigid topologies on groups]
{Rigid topologies on groups}

\author[Eli Glasner]{Eli Glasner}
\address{Department of Mathematics,
Tel-Aviv University, Ramat Aviv, Israel}
\email{glasner@math.tau.ac.il}

\author[Benjamin Weiss]{Benjamin Weiss}
\address{Mathematics Institute, Hebrew University of Jerusalem,
Jerusalem, Israel}
\email{weiss@math.huji.ac.il}

%


\setcounter{secnumdepth}{2}



\setcounter{section}{0}


\begin{abstract}
Our main result is to show that 
every infinite, countable, residually finite group $G$ admits a Hausdorff
group topology which is neither discrete nor precompact.
\end{abstract}
 
%


\keywords{Residually finite groups, discrete topology, precompact topology, rigidity, uniform rigidity}

\thanks{{
AMS classification:
Primary 20E26; Secondary 37B20}}

\begin{date}
{December 14, 2022}
\end{date}

\maketitle


\setcounter{secnumdepth}{2}


\setcounter{section}{0}


\section*{Introduction}

Let $G$ be an infinite countable group.
It is an old problem to determine whether or not $G$ admits  a non-discrete Hausdorff topology
in which the group operations are continuous (i.e. a Hausdorff group topology).
It was shown by Olshanskii \cite{Ol-80} that there are infinite countable groups for which
the discrete topology is the only possible group topology. Such groups are called
{\em non-topologizable}.

Let $G$ be an infinite countable group. The collection of normal finite index sub-groups 
$N < G$, forms a basis of neighborhoods of the identity for a group topology on $G$ 
(i.e. a topology with respect to which $G$ is a topological group) called the {\em profinite topology}. 
This topology is Hausdorff  iff $G$ is residually finite. The completion of $G$ with respect to this
topology (or rather with respect to the corresponding uniform structure) is a compact topological group, 
hence the profinite topology is precompact.
%

Free groups, finitely generated nilpotent groups,
polycyclic-by-finite groups, finitely generated linear groups, and fundamental groups
of compact $3$-manifolds are all residually finite. More specifically, all the groups
$SL(n, \Z)$ are residually finite.
An example of a non-residually finite countable group is $\Q$ -- the group of rational numbers.
It is a divisible group and thus has no finite index subgroups.

In this note we show that every infinite, countable, residually finite group admits a
Hausdorff group topology which is neither discrete nor precompact. The main idea of
the proof is to construct a compact metric, free dynamical system $(X,G)$ such that
the topology on $G$ inherited from the Polish group $\Homeo (X)$ is the required topology.
The construction of the dynamical system $(X,G)$ is inspired by that of the examples in \cite{KW-81}.

A (topological) {\em $G$-dynamical system} is a pair $(X,G)$ where $X$ is a compact metric
space and $G$ acts on $X$ by homeomorphisms; i.e. via a homomorphism  $\theta$ from $G$ into
the Polish group $\Homeo (X)$ (equipped with the uniform convergence topology). We
say that the action of $G$ on $X$ is {\em uniformly rigid} when $\theta$ is injective and the topology
induced on $G$ via this homomorphism is not discrete, or in other words, when there
is a sequence $g_n \in G \setminus \{e\}$  such that the sequence $\theta(g_n)$ tends uniformly to $\id_X$, the
identity map on $X$ (see \cite{GM-89} for more details).
We call such a sequence {\em a rigidity sequence}.

On the way to proving our main result we also show that for amenable groups 
that are {\em maximally almost periodic} (maxAP for short) , i.e. admit an embedding into a compact topological group,
there is a Hausdorff group topology 
which is neither discrete nor precompact. For this we use tools from ergodic theory.

The term {\em rigid} originates in ergodic theory in the work of H. Furstenberg and the second author
\cite{FW-77} (see also \cite{W-72}).
The topological analogue was studied by the first author and D. Maon in \cite{GM-89}. 

\br

\section{The amenable case}

Let  $\Aut(\mu)$ be the group of automorphisms of the standard Lebesgue probability measure space 
$(X,\mu) = ([0,1], \mu)$. This group has a Polish topology induced by the strong operator topology on the 
unitary operators defined in $L^2(X,\mu)$ by elements of $\Aut(\mu)$.

A measure theoretical measure preserving $G$-dynamical system 
$\bX = (X, \Xcal, \mu, G)$, is defined by a homomorphism $\rho$  from $G$ into
 $\Aut(\mu)$. The system $\bX$ is {\em free} when for each $e \not=g \in G$ the set of $g$-fixed points has
measure zero. The system $\bX$ is {\em rigid} when $\rho$ is injective and the topology induced on
$\rho(G)$ from $\Aut(\mu)$ is not discrete; that is, there is a sequence $g_n \in G \setminus \{e\}$ such that the
sequence $\rho(g_n)$ tends to $\id_X$ in $\Aut(\mu)$.

\br

Recall that an infinite (discrete) group $G$ is {\em maximally almost periodic} (maxAP group for short) if there is
an injective homomorphism $\eta : G \to K$, with $K$ a compact Hausdorff topological group.
When such a monomorphism exists we can assume that $\eta(G)$ is dense in $K$.
It then follows that the dynamical system $(K,\la_K,G)$, with $\la_K$ being the Haar measure on $K$ 
and the action is defined by left multiplication via $\eta$,
is a free measure preserving, ergodic, rigid $G$-action.

\begin{prop}\label{amen}
Every infinite, countable, 
maximally almost periodic, amenable group $G$ admits a Hausdorff group topology 
which is neither discrete nor precompact.
\end{prop}

\begin{proof}
Consider the Polish space $\A(G)$ of measure preserving $G$-actions,
which consists of the representations of $G$ in $\Aut (\mu)$ (for more details see e.g. \cite[Chapter II]{Ke-10}).
The group $\Aut(\mu)$ acts by conjugations on $\A(G)$, and when $G$ is amenable the orbit of any free
action is dense in $\A(G)$ (see \cite{FoW-04}).

It is well known that (for an arbitrary infinite countable group $G$) the
set $\Fcal$ of free actions forms a $G_\del$ subset of $\A(G)$ (see \cite{GK-98}). It is also well known 
(and easy to see) that (again for every infinite countable $G$) the set $\Rcal$ of rigid actions is a 
$G_\del$ subset of $\A(G)$ (see e.g. \cite{Sa-09}). 

As observed above, since we assume that $G$ is infinite and 
maximally almost periodic,
$G$ admits a natural rigid free measure preserving action, namely 
the action on the compact embedding which it admits by definition.

It follows that the set $\Rcal$ is nonempty, hence by the fact that, for amenable $G$, the
orbit of an ergodic free $\rho \in \A(G)$ is dense in $\A(G)$, it is a dense $G_\del$ subset of $\A(G)$.
Finally, as we assume that $G$ is amenable, the set $\Wcal$, of weakly mixing actions, forms
a dense $G_\del$ subset of $\A(G)$ (again see \cite{Ke-10}).
It then follows that the collection $\Rcal \cap \Wcal \cap \Fcal$ is nonempty and for each element $\rho$ of
this set the topology inherited from $\Aut(\mu)$ onto $\A(G)$ yields the required topology
on $G$. Indeed, the fact that $\rho$ is free implies that this topology is Hausdorff (in fact
metrizable), the fact that  $\rho$ is rigid implies that this topology is not discrete, and the
fact that  $\rho$ is weakly mixing implies that this topology is not precompact. Finally, as
it is inherited from $\Aut(\mu)$, it is clearly a group topology.
\end{proof}

\br

\begin{rmk}

The same proof will work for any infinite, countable, 
maxAP group $G$ for which:
\begin{enumerate}
\item
The orbit of a
free precompact action  is dense in $\A(G)$.
\item
$\Wcal$ is residual in $\A(G)$.
\end{enumerate}
By a result of Kerr and Pichot, \cite{KP-08} (2) holds iff $G$ does not have Kazhdan's property (T). 
In  \cite{Ke-12} Kechris shows that (1) holds for the free groups $F_n$.
\end{rmk}

\br

\section{The main theorem}

\begin{thm}\label{main}
Every infinite, countable, residually finite group $G$ admits a Hausdorff
group topology which is neither discrete nor precompact.
\end{thm}

\begin{proof}

{\bf Case I: $G$ is locally finite.}

Clearly,  every countable locally finite group $G$ is a countable union of an increasing 
sequence of finite groups which then form a F{\o}lner sequence for $G$.
Thus such a group is amenable
{\footnote{Note that a  countable locally finite group need not be residually finite.
E.g. the group $A_\infty(\N) < S_\infty(\N)$ consisting of the alternating 
finitely supported permutations of $\N$, is both locally finite and simple.}}.
Since we assume that $G$ is residually finite, its pro-finite completion
shows that it is maxAP and  our claim follows from Proposition \ref{amen}.

\br

{\bf Case II: $G$ is not locally finite.}
 
Let $\{B_n\}_{n \in \N}$ be a sequence of finite symmetric sets such that
 \begin{enumerate}
\item
$\{e\} = B_0 \subset B_1 \subset B_2 \subset \cdots \subset B_n \subset B_{n+1} \subset \cdots$
\item
$\bigcup_{n \in \N} B_n = G$.
 \end{enumerate}

Let $\Om =  [0,1]^G$  and define the action of $G$ on $\Om$, as usual, by
$$
(\sig_h \om)(g) = \om(gh), \qquad g,h \in G.
$$
We will construct inductively periodic sequences of elements $\phi_n, \psi_n \in \Om$, 
an element $\xi = \lim \psi_n \in \Om$ and a sequence $g_n \in G$
such that, for the action of $G$ restricted to the
subsystem $X = \ol{G\xi}$, the sequence $g_n$ will tend uniformly to $\id_X$. 
We will then show that the uniform convergence topology induced on $G$, considered as a subset of
$\Homeo (X)$, satisfies the required properties.

\br

{\bf Step 1:}

Let $S_1$ be a symmetric finite set containing $e$, the unit element of $G$, such that $\langle S_1 \rangle$,
the subgroup generated by $S_1$, is infinite. Let $H_1$ be a finite index normal subgroup
of $G$ such that
$$
S_1^3 \cap H_1 = \{e\}.
$$
Define $\phi_1 : G \to [0,1]$ by
$$
\phi_1(g) =
\begin{cases}
1 & g \in S_1 H_1\\
0 & g \not \in S_1 H_1.
\end{cases}
$$
Clearly $\phi_1(gh) = \phi_1(g)$ for all $g \in G$ and $h \in H_1$.

\br

{\bf Step 2:}

Let $F_1$ be a finite symmetric subset in $H_1 \setminus \{e\}$,
and set $S_2 = S_1\cup F_1 \cup B_1$.
Let $H_2 < H_1$ be a finite index normal subgroup of $G$ such that
$$
S_2^{30} \cap H_2 = \{e\}.
$$

Define $\phi_2 : G \to [0,1]$ by
$$
\phi_2(g) =
\begin{cases}
1 & g \in H_2\\
\frac{10}{11} & g \in (S_2 \setminus \{e\}) H_2\\
\frac{9}{11} & g \in (S^2_2 \setminus S_2) H_2\\
\frac{8}{11} & g \in (S^3_2 \setminus S^2_2) H_2\\
 &\cdots\\
\frac{1}{11} & g \in (S^{10}_2 \setminus S_2^9) H_2\\
0 & g \not \in S_2^{10} H_2.
\end{cases}
$$
Note that since $\langle S_1 \rangle$ is infinite, all of the sets $S_2^{k+1} \setminus S_2^k$ are non-empty.

With this definition we clearly have, for every $s \in S_2$ (and in particular $s \in F_1$)
and $g \in G$
$$
| \phi_2(gs) - \phi_2(g) | < \frac{1}{10}.
$$
Next set
$$
\psi_2 = \min(\phi_1, \phi_2)
$$
and observe that also
\begin{equation}\label{ine1}
| \psi_2(gs) - \psi_2(g) | < \frac{1}{10}
\end{equation}
for all $s \in F_1$ and $g \in G$.
Clearly $\psi_2(gh) = \psi_2(g)$ for all $g \in G$ and $h \in H_2$.

\br

{\bf Step 3:}

Fix now $F_2$,  a finite symmetric subset in $H_2 \setminus \{e\}$, and set $S_3 = S_2 \cup F_2 \cup B_2$.
Let $H_3 < H_2$ be a finite index normal subgroup of $G$ such that
$$
S_3^{300} \cap H_3 = \{e\}.
$$
Define
$\phi_3 : G \to [0,1]$ by
$$
\phi_3(g) =
\begin{cases}
1 & g \in H_3\\
\frac{100}{101} & g \in (S_3 \setminus \{e\}) H_3\\
\frac{99}{101} & g \in (S^2_3 \setminus S_3) H_3\\
\frac{98}{101} & g \in (S^3_3 \setminus S^2_3) H_3\\
 &\cdots\\
\frac{1}{101} & g \in (S^{100}_3 \setminus S_3^{99}) H_3\\
0 & g \not \in S_3^{100} H_3.
\end{cases}
$$
Clearly $\phi_3(gh) = \phi_2(g)$ for all $g \in G$ and $h \in H_3$.

Now for $s \in S_3$ and $g \in G$,
\begin{equation}\label{ine2}
| \phi_3(gs) - \phi_3(g) | < \frac{1}{100},
\end{equation}
and with
$$
\psi_3 = \min(\psi_2,\phi_3)
$$
we get
$$
| \psi_3(gs) - \psi_3(g) | <  \frac{1}{100},
$$
for $s \in F_2$, since $F_2 \subset H_2$, $\psi_2$ is $H_2$ invariant and $F_2 \subset S_3$.

Furthermore, the inequality 
\begin{equation*}
| \psi_3(gs) - \psi_3(g) | < \frac{1}{10} +  \frac{1}{100},
\end{equation*}
holds for 
$s \in F_1$ since $F_1 \subset S_3$,
and $\phi_3$ satisfies (\ref{ine2}) while, $\psi_2$ satisfies (\ref{ine1}).

\br

Now the general picture is as follows: The normal finite index subgroups
$$
G = H_0 >  H_ 1 >  H_2  >  \cdots  >  H_n >  H_{n+1} > \cdots
$$
are built inductively together with the finite sets $F_n  \subset H_n \setminus \{e\}$. 
The functions  $\psi_n$ will satisfy, for
all $s \in  F_i$ with $i < n$:
$$
|\psi_n(gs) - \psi_n(g) | <\frac{1}{10^i} +    \cdots +  \frac{1}{10^n}.
$$
Hence, this will propagate to the limit of the monotone decreasing sequence
$$
\xi = \lim \psi_n.
$$
Note that for all $n$, $\psi_n(0) =1$, hence also $\xi(0)=1$.

\br

As promised we now define $X$  to be the orbit closure of $\xi \in \Om$ under $G$:
$$
X = \ol{ \xi G}
$$
and conclude that the dynamical system $(X,G)$ is uniformly rigid.
In fact, every sequence 
$\{s_n\}_{n \in \N}$ with $s_n \in F_n$ will converge uniformly on $X$ to the identity.

\br

Next let us observe that the configuration ${\bf{0}} \in \Om$ 
 belongs to $X$ and that the set $\{{\bf{0}}\}$ is
the unique minimal subset of $X$. This follows from the fact that for every $n$ we have
 $H_n \cap S_n^{3 \times 10^n} = \{e\}$,
which means that when we define $\psi_n$ it will vanish on a syndetic
set of translates of $B_{n-1}$. By the definition of  $\psi_{n+1}$, as minimum of  $\psi_n$ and $\phi_{n+1}$, these
zero regions will persist in $\xi$.

\br

To finish the proof we now observe that if the topology induced on $G$ from $\Homeo (X)$
were precompact, then the system $(X,G)$ would have been equicontinuous, and being
by definition point-transitive, it would have necessarily been minimal.
\end{proof}

\br

\begin{rmk}
The existence of non-topologizable infinite countable groups (see \cite{Ol-80}
and, the more recent \cite{KOO-13}) shows that the claim of the theorem can not 
hold for all infinite countable groups.
\end{rmk}


\begin{question}
Is there an infinite countable maxAP group $G$ such that
every non-discrete Hausdorff topology on $G$ is precompact?
\end{question}

\br

\begin{prop}
A finitely generated maxAP group $G$ is residually finite.
\end{prop}

\begin{proof}
Suppose $G$ is a subgroup of the compact group $K$.
By the Peter-Weyl theorem $K$ is the inverse limit of compact Lie groups (connected or not).
Hence, for any element $g \in G$ there is some Lie group $K_0$ and a homomorphism $\pi :K \to K_0$ such
that $\pi(g)$ is not the identity. Now the image of $G$ under $\pi$ is
a finitely generated linear group and hence is residually finite,
so that there is a homomorphism $\phi$ from $\pi(G)$ to a finite group $L$
such that $\phi(\pi(g))$ is not the identity. This shows that
$G$ is residually finite.
\end{proof}

\br

\begin{prop}
An infinite countable group $G$ that admits an embedding into a compact totally disconnected group $K$
 is residually finite.
\end{prop}

\begin{proof}
Suppose $G$ is a subgroup of the compact totally disconnected group $K$.
By van Dantzig's theorem \cite{vD-36} the collection of normal compact open subgroups of $K$
forms a basis for the topology at the identity of $K$. If $N$ is such a group then 
$K/N$ is a finite group and, as the maps $\pi_N : K \to K/N$ separate points on $G$,
it follows that $G$ is residually finite.
\end{proof}

\br

\section{Some further remarks}

The possibility to choose, in the proof of Theorem \ref{main}, the finite sets $F_n \subset H_n$
arbitrarily, results in great flexibility when one considers further properties of the
generated topology. This flexibility is demonstrated in the following two propositions.

\br

\begin{prop}
In the context of Theorem \ref{main} suppose that $g_0 \in G$ is an element
of infinite order. Then it can be arranged that there is a sequence $n_k \in \N$
 such that, in the constructed topology, the sequence $g_0^{n_k}$
will tend uniformly to the identity.
\end{prop}

\begin{proof}
 If $H$ is a finite index normal subgroup of $G$ and $g_0$ has infinite order then some
power of $g_0$ lies in $H$. In fact, since there are only finitely many $H$-cosets for some
$k < m$ we have $g_0^k H = g_0^m H$
and then $g_0^{m-k}$
lies in $H$. Thus, in the proof of Theorem \ref{main}
we can choose, for each $n$ some $k_n$ so that $g_0^{k_n}$
lies in $H_n$, and include it in $F_n$.
 \end{proof}

\br

In \cite{AGW-07} the notion of {\em cones} was defined for a finitely generated group.
They were defined in order to provide `a sense of direction' in a finitely generated group.
For more details on these objects we refer the reader to \cite{AGW-07}.

\begin{prop}
 Let $G$ be an infinite, countable, residually finite, finitely generated group. 
 Then, for any list of countably many cones $C_i, \  i =1,2,\dots,$ there is a non-discrete, 
 not precompact, Hausdorff group topology on $G$ such that for every $i$ the cone
$C_i$ contains a rigidity sequence.
 \end{prop}

\begin{proof}
Every cone contains arbitrarily large balls and thus meets every subgroup $H_n$,
as in the proof of Theorem \ref{main}. Thus it is possible to choose the finite sets $F_n  \subset H_n$
in such a way that for every $i,\  F_n \cap  C_i \not= \emptyset$ infinitely often.
\end{proof}

\br

The following question is raised in \cite{Sa-09}: 

\begin{question}\label{sam}
Given a infinite countable discrete amenable group $G$, does there exist a free 
action of $G$ on a Lebesgue probability space which is both rigid and weakly mixing?
\end{question}

As far as we know this question is still open.
Note that a positive answer would show that every infinite countable discrete amenable group
admits a topology which is neither discrete nor precompact, so in particular, topologizable.
Denis Osin informed us that as far as he knows, there is no known example of a 
non-topologizable, countable, amenable group.

We also note that in the proof of Proposition \ref{amen} we get the required topology from a measure preserving system.
However, the existence of such a topology does not imply that 
there is a measure preserving system which is rigid. Thus exhibiting a non-discrete 
non-precompact topology for an amenable group will not automatically answer 
Question \ref{sam}.

\br

%
%

\br

\end{document}